\documentclass{article}

\usepackage[a4paper,top=2.54cm,bottom=2.54cm,left=3.17cm,right=3.17cm,%
            includehead,includefoot]{geometry}

\usepackage{amsmath,amssymb,amsfonts,amsthm}
\usepackage{graphicx}
\usepackage{float}
\usepackage{xcolor}
\usepackage[numbers,square,sort&compress]{natbib}
\usepackage{hyperref}
  \hypersetup{colorlinks,citecolor=blue,linkcolor=blue,breaklinks=true}
\usepackage{epstopdf}
\usepackage{amsthm,amsmath,amssymb}

\usepackage{mathrsfs}

\makeatletter

\newcommand{\Rmnum}[1]{\expandafter\@slowromancap\romannumeral #1@}
\makeatother

\usepackage{yhmath} 

\allowdisplaybreaks

\newtheorem{theorem}{Theorem}[section]
\newtheorem{lemma}{Lemma}[section]
\newtheorem{corollary}[theorem]{Corollary}
\newtheorem{proposition}[theorem]{Proposition}

\newcommand{\bbm}{\begin{bmatrix}}
\newcommand{\ebm}{\end{bmatrix}}

\begin{document}

\title{\uppercase{the logarithmic sobolev inequality for a submanifold in manifold with nonnegative sectional curvature }}

\author{
  Chengyang Yi
    \thanks{School of Mathematical Sciences, East China Normal University, 500 Dongchuan Road, Shanghai 200241,
P. R. of China E-mail address: 52195500013@stu.ecnu.edu.cn. }
  \and
  Yu Zheng
    \thanks{School of Mathematical Sciences, East China Normal University, 500 Dongchuan
Road, Shanghai 200241, P. R. of China E-mail address: zhyu@math.ecnu.edu.cn.}
}

\date{}

\maketitle

\begin{abstract}
We prove a sharp logarithmic Sobolev inequality which
holds for compact submanifolds without boundary in Riemannian manifold with nonnegative sectional curvature of arbitrary dimension and
codimension, while the ambient manifold needs to have a specific Euclid-like property. Like the Michael-Simon Sobolev inequality, this inequality
includes a term involving the mean curvature.  This extends a recent result of S. Brendle with Euclidean setting.
\end{abstract}

\section{Introduction}

In 2019, S. Brendle \cite{1} proved a Sobolev inequality which holds on submanifolds in Euclidean space of arbitrary dimension and codimension. The inequality is sharp if the codimension is at most 2. Soon, he \cite{2} proved a sharp logarithmic Sobolev inequality which holds on submanifolds in Euclidean space of arbitrary dimension and codimension at the same year. In 2020, he \cite{3} extended the result of the Sobolev inequality to Riemannnian manifolds with nonnegative curvature which gives the asymptotic volume ratio due to the Bishop-Gromov volume comparison theorem. Inspired by \cite{3}, we extend the result of the logarithmic Sobolev inequality to ambient Riemannian manifolds with nonnegative sectional curvature under an assumption.

Let $M$ be a complete noncompact Riemannian manifold of dimension $k$. We say $M$ satisfies the condition (P), if there is a point $p\in M$ such that the following limit exists and is positive.
$$
\lim_{r\rightarrow\infty}  \left((4\pi)^{-\frac{k}{2}} r^{-k}\displaystyle\int_{M}e^{-\frac{d(x,p)^{2}}{4r^{2}}}d\mathrm{vol}(x)  \right)  .
$$
We denotes the limit by $\theta$. Note that the limit is equal to 1 when $M=\mathbb{R}^{k}$. So we call it a specific Euclid-like property. We have the following result
\begin{theorem}	\label{th:01}
Let $M$ be a complete noncompact Riemannian manifold of dimension $n+m$ with nonnegative sectional curvature and satisfies the condition (P). Let $\Sigma$ be a compact $n$-dimension submanifold of $M$ without boundary, and let $f$ be a positive smooth function on $\Sigma$. Then
\begin{equation*}
\begin{array}{lllll}
  &\displaystyle\int_{\Sigma}f\left(\log f+n+\frac {n}{2}\log(4\pi)+\log \theta \right)d\mathrm{vol}-\int_{\Sigma}\frac {|\nabla^{\Sigma}f|^{2}}{f}d\mathrm{vol}-\int_{\Sigma}f|H|^{2}d\mathrm{vol}\\
  &\leq\displaystyle\left(\int_{\Sigma}f d\mathrm{vol}\right)\log {\left(\int_{\Sigma}f d\mathrm{vol}\right)},
\end{array}
\end{equation*}
where $H$ denotes the mean curvature vector of $\Sigma$.
\end{theorem}

Recall the paraboloid of revolution $\Gamma:\overrightarrow{r}(u,v)=(u\cos v,u\sin v,\frac{1}{2}au^{2})$ with $u\in[0,\infty)$ and $v\in[0,2\pi)$, where $a$ is a positive constant. We find that it also satisfies the condition (P). Moreover $\theta =2$ which is independent on $a$. It's amazing! And we have the result
\begin{corollary}\label{th:02}
Let $\gamma$ is a smooth closed curve on $\Gamma$, and let $f$ be a positive smooth function on $\gamma$. Then
\begin{equation*}
\begin{array}{lllll}
  &\displaystyle\int_{\gamma}f\left(\log f+1+\frac {1}{2}\log(4\pi)+\log 2 \right)d\mathrm{vol}-\int_{\gamma}\frac {|\nabla^{\gamma}f|^{2}}{f}d\mathrm{vol}-\int_{\gamma}f|\kappa_{n}|^{2}d\mathrm{vol}\\
  &\leq\displaystyle\left(\int_{\gamma}f d\mathrm{vol}\right)\log {\left(\int_{\gamma}f d\mathrm{vol}\right)},
\end{array}
\end{equation*}
where $\kappa_{n}$ denotes the geodesic normal curvature of $\gamma$ with respect to $\Gamma$.
\end{corollary}
However, cylinder $S^{1}\times \mathbb{R}$ doesn't satisfies the condition (P) in spite of $\theta=0$.

The logarithmic Sobolev inequality has been studied by numerous authors (see e.g.\cite{6,7,8,9,10}). Our proof of theorem 1.1 is in the spirit of ABP-techniques in \cite{2}. ABP-techniques have been applied to various classes of linear and nonlinear elliptic equations in the Euclidean space for a long time. Due to some difficulties, it was not until 1997 that Cabr\'{e} \cite{4} developed them to Riemannnian manifolds.

\section{Preliminaries}

Let's talk about the condition (P) first.

\begin{proposition}\label{th:01}
  Let $M$ be a complete noncompact Riemannian manifold of dimension $k$, then the following are equivalent:
  \begin{itemize}
    \item[(a)] There exists a point $p\in M$ such that the limit
$$
\lim_{r\rightarrow\infty}  \left((4\pi)^{-\frac{k}{2}} r^{-k}\displaystyle\int_{M}e^{-\frac{d(x,p)^{2}}{4r^{2}}}d\mathrm{vol}(x)  \right)
$$
exists.
    \item[(b)] The limit
$$
\lim_{r\rightarrow\infty}  \left((4\pi)^{-\frac{k}{2}} r^{-k}\displaystyle\int_{M}e^{-\frac{d(x,p)^{2}}{4r^{2}}}d\mathrm{vol}(x)  \right)
$$
exists for every $p\in M$.
    \item[(c)] For any compact subset $K\subset M$ and for any Borel map $p:M\rightarrow K$, the limit
$$
\lim_{r\rightarrow\infty}  \left((4\pi)^{-\frac{k}{2}} r^{-k}\displaystyle\int_{M}e^{-\frac{d(x,p(x))^{2}}{4r^{2}}}d\mathrm{vol}(x)  \right)
$$
exists.
  \end{itemize}
  Moreover, both of these limits are the same one if exist.
\end{proposition}
\begin{proof}
Clearly, $(b)\Rightarrow(a)$ and $(c)\Rightarrow(b)$ are trivial. It remains to show $(a)\Rightarrow(c)$. We assume that there exists a point $p_{0}\in M$ such that the limit
$$
\lim_{r\rightarrow\infty}  \left((4\pi)^{-\frac{k}{2}} r^{-k}\displaystyle\int_{M}e^{-\frac{d(x,p_{0})^{2}}{4r^{2}}}d\mathrm{vol}(x)  \right)
$$
exists and equals to $\theta$. Given a compact subset $K\subset M$ and a Borel map $p:M\rightarrow K$. We define a positive constant
$$
C:=\mathrm{sup}\{ d(p_{0},p(x)):x\in M \}.
$$
For any fixed $\varepsilon>0$ sufficiently small, note that
\begin{equation*}
\begin{array}{lllll}
-\frac{d(x,p(x))^{2}}{4r^{2}}&=-\frac{d(x,p_{0})^{2}}{4r^{2}}\cdot\frac{d(x,p(x))^{2}}{d(x,p_{0})^{2}}\\
&\geq-\frac{d(x,p_{0})^{2}}{4r^{2}}\left(1+\frac{d(p_{0},p(x))}{d(x,p_{0})}\right)^{2}\\
&\geq-\frac{d(x,p_{0})^{2}}{4r^{2}}\left(1+C\varepsilon\right)^{2}\\
\end{array}
\end{equation*}
for all $x\in M\setminus B_{\varepsilon^{-1}}(p_{0})$. Similarly, we have
$$
-\frac{d(x,p(x))^{2}}{4r^{2}}\leq-\frac{d(x,p_{0})^{2}}{4r^{2}}\left(1-C\varepsilon\right)^{2}
$$
for all $x\in M\setminus B_{\varepsilon^{-1}}(p_{0})$. Thus,
\begin{equation*}
\begin{array}{lllll}
r^{-k}\displaystyle\int_{M\setminus B_{\varepsilon^{-1}}(p_{0})}e^{-\frac{d(x,p_{0})^{2}}{4r^{2}}\left(1+C\varepsilon\right)^{2}}d\mathrm{vol}(x)&\leq r^{-k}\displaystyle\int_{M\setminus B_{\varepsilon^{-1}}(p_{0})}e^{-\frac{d(x,p(x))^{2}}{4r^{2}}}d\mathrm{vol}(x)\\
&\leq r^{-k}\displaystyle\int_{M\setminus B_{\varepsilon^{-1}}(p_{0})}e^{-\frac{d(x,p_{0})^{2}}{4r^{2}}\left(1-C\varepsilon\right)^{2}}d\mathrm{vol}(x).
\end{array}
\end{equation*}
It's easy to see that
$$
\lim_{r\rightarrow\infty}  \left((4\pi)^{-\frac{k}{2}} r^{-k}\displaystyle\int_{B_{\varepsilon^{-1}}(p_{0})}e^{-\frac{d(x,p_{0})^{2}}{4r^{2}}}d\mathrm{vol}(x)  \right)=0
$$
and
$$
\lim_{r\rightarrow\infty}  \left((4\pi)^{-\frac{k}{2}} r^{-k}\displaystyle\int_{B_{\varepsilon^{-1}}(p_{0})}e^{-\frac{d(x,p(x))^{2}}{4r^{2}}}d\mathrm{vol}(x)  \right)=0.
$$
From the assumption, we have
$$
\lim_{r\rightarrow\infty}  \left((4\pi)^{-\frac{k}{2}} r^{-k}\displaystyle\int_{M\setminus B_{\varepsilon^{-1}}(p_{0})}e^{-\frac{d(x,p_{0})^{2}}{4r^{2}}\left(1+C\varepsilon\right)^{2}}d\mathrm{vol}(x)  \right)=\left(1+C\varepsilon\right)^{-k}\theta
$$
and
$$
\lim_{r\rightarrow\infty}  \left((4\pi)^{-\frac{k}{2}} r^{-k}\displaystyle\int_{M\setminus B_{\varepsilon^{-1}}(p_{0})}e^{-\frac{d(x,p_{0})^{2}}{4r^{2}}\left(1-C\varepsilon\right)^{2}}d\mathrm{vol}(x)  \right)=\left(1-C\varepsilon\right)^{-k}\theta.
$$
Combining with
\begin{equation*}
\begin{array}{lllll}
&(4\pi)^{-\frac{k}{2}}r^{-k}\displaystyle\int_{M\setminus B_{\varepsilon^{-1}}(p_{0})}e^{-\frac{d(x,p_{0})^{2}}{4r^{2}}\left(1+C\varepsilon\right)^{2}}d\mathrm{vol}(x)+(4\pi)^{-\frac{k}{2}} r^{-k}\displaystyle\int_{B_{\varepsilon^{-1}}(p_{0})}e^{-\frac{d(x,p(x))^{2}}{4r^{2}}}d\mathrm{vol}(x)\\
&\leq (4\pi)^{-\frac{k}{2}} r^{-k}\displaystyle\int_{M}e^{-\frac{d(x,p(x))^{2}}{4r^{2}}}d\mathrm{vol}(x)  \\
&\leq(4\pi)^{-\frac{k}{2}}r^{-k}\displaystyle\int_{M\setminus B_{\varepsilon^{-1}}(p_{0})}e^{-\frac{d(x,p_{0})^{2}}{4r^{2}}\left(1-C\varepsilon\right)^{2}}d\mathrm{vol}(x)+(4\pi)^{-\frac{k}{2}} r^{-k}\displaystyle\int_{B_{\varepsilon^{-1}}(p_{0})}e^{-\frac{d(x,p(x))^{2}}{4r^{2}}}d\mathrm{vol}(x),
\end{array}
\end{equation*}
(c) follows from a standard $\varepsilon-\delta$ discussion.

\end{proof}

\section{Proof of Theorem 1.1}

Recall the definition of the second fundamental form $\Rmnum{2}$ of $\Sigma$ with respect to $M$:
$$
\langle\Rmnum{2}(X,Y),V\rangle=\langle\ \bar{D}_{X}Y,V\rangle=-\langle\ \bar{D}_{X}V,Y\rangle,
$$
where $X,Y$ are tangent vector fields, $V$ is a normal vector field and $\bar{D}$ denotes the connection on $M$. Moreover, the mean curvature vector $H$ is defined as the trace of the second fundamental form $\Rmnum{2}$.

We now give the proof of Theorem 1.1. We first consider the special case that $\Sigma$ is connected. By scaling, we may assume that
\begin{equation*}
  \int_{\Sigma}f\log f d\mathrm{vol}-\int_{\Sigma}\frac {|\nabla^{\Sigma}f|^{2}}{f}d\mathrm{vol}-\int_{\Sigma}f|H|^{2}d\mathrm{vol}=0.
\end{equation*}
From functional analysis and standard elliptic theory, we can find a smooth function $u:\Sigma\rightarrow\mathbb{R}$ such that
\begin{equation*}
  \mathrm{div}_{\Sigma}\left(f\nabla^{\Sigma}u\right)=f\log f-\frac {|\nabla^{\Sigma}f|^{2}}{f}-f|H|^{2}.
\end{equation*}
In the following, we fix a positive number $r$. We denote the contact set
\begin{equation*}
  A=\left\{(\bar {x},\bar {y})\in T^{\perp}{\Sigma}:ru(x)+\frac{1}{2}d\left(x,\mathrm{exp}_{\bar {x}}\left(r\nabla^{\Sigma}u(\bar{x})+r\bar{y}\right)\right)^{2}\geq ru(\bar{x})+\frac{1}{2}r^{2}\left(|\nabla^{\Sigma}u(\bar{x})|^{2}+|\bar{y}|^{2}\right),\forall x\in \Sigma\right\}.
\end{equation*}
Moreover, we define a map $\Phi:T^{\perp}\Sigma\rightarrow M$ by
\begin{equation*}
  \Phi(x,y)=\mathrm{exp}_{x}\left(r\nabla^{\Sigma}u(x)+ry\right)
\end{equation*}
for all $(x,y)\in T^{\perp}\Sigma$.

\begin{lemma}	Suppose that $(\bar{x},\bar{y})\in A$, then
\begin{equation*}
  d\left(\bar{x},\Phi (\bar{x},\bar{y})\right)^{2}=r^{2}\left(|\nabla^{\Sigma}u(\bar{x})|^{2}+|\bar{y}|^{2}\right).
\end{equation*}
\end{lemma}

\begin{proof}
Let $\bar {\gamma}(t):=\mathrm{exp}_{\bar {x}}\left(rt\nabla^{\Sigma}u(\bar{x})+rt\bar{y}\right)$ for $t\in[0,1]$. From the definition of $A$, we have
\begin{equation*}
  ru(\bar{x})+\frac{1}{2}d\left(\bar{x},\mathrm{exp}_{\bar {x}}\left(r\nabla^{\Sigma}u(\bar{x})+r\bar{y}\right)\right)^{2}\geq ru(\bar{x})+\frac{1}{2}r^{2}\left(|\nabla^{\Sigma}u(\bar{x})|^{2}+|\bar{y}|^{2}\right).
\end{equation*}

Thus, $ d\left(\bar{x},\Phi (\bar{x},\bar{y})\right)^{2}\geq r^{2}\left(|\nabla^{\Sigma}u(\bar{x})|^{2}+|\bar{y}|^{2}\right).$
On the other hand,
\begin{equation*}
\begin{array}{lllll}
r^{2}\left(|\nabla^{\Sigma}u(\bar{x})|^{2}+|\bar{y}|^{2}\right)&=|\bar{\gamma}'(0)|^{2}\\
&=\left(\displaystyle\int_0^{1}|\bar{\gamma}'(t)|dt\right)^{2}\\
&\geq d\left(\bar{x},\Phi (\bar{x},\bar{y})\right)^{2}.
\end{array}
\end{equation*}
Then, the lemma follows.
\end{proof}

\begin{lemma}	$\Phi(A)=M$.
\end{lemma}
\begin{proof}
Fix a point $p\in M$. Since $\Sigma$ is compact without boundary, the function $x\mapsto ru(x)+\frac{1}{2}d(x,p)^{2}$ must attain its minimum at some point denoted by $\bar{x}$ on $\Sigma$. Moreover, we can find a minimizing geodesic $\bar{\gamma} :[0,1]\rightarrow M$  such that $\bar{\gamma}(0)=\bar{x}$ and $\bar{\gamma}(1)=p$. For every path $\gamma :[0,1]\rightarrow M$ satisfying $\gamma (0)\in \Sigma$ and $\gamma(1)=p$, we obtain
\begin{equation*}
\begin{array}{lllll}
ru(\gamma(0))+E(\gamma)&\geq ru(\gamma(0))+\frac {1}{2}d(\gamma(0),p)^{2}\\
&\geq ru(\bar{x})+\frac {1}{2}d(\bar{x},p)^{2}\\
&=ru(\bar{\gamma}(0))+\frac{1}{2}|\bar{\gamma}'(0)|^{2}\\
&=ru(\bar{\gamma}(0))+E(\bar{\gamma}),
\end{array}
\end{equation*}
where $E(\gamma)$ denotes the energy of $\gamma$. In other words, the path $\gamma$ minimizes the functional $ru(\gamma(0))+E(\gamma)$ among all paths $\bar{\gamma} :[0,1]\rightarrow M$ satisfying $\gamma (0)\in \Sigma$ and $\gamma(1)=p$. Hence, the formula for the first variation implies
\begin{equation*}
  \bar{\gamma}'(0)-r\nabla ^{\Sigma}u(\bar{x})\in T_{\bar{x}}^{\perp}\Sigma.
\end{equation*}
Consequently, we can find a vector $\bar{y}\in T_{\bar{x}}^{\perp}\Sigma$ such that
\begin{equation*}
  \bar{\gamma}'(0)=r\nabla ^{\Sigma}u(\bar{x})+r\bar{y}.
\end{equation*}
It remains to show $(\bar{x},\bar{y})\in A$. For each point $x\in \Sigma$, we have
\begin{equation*}
\begin{array}{lllll}
ru(x)+\frac{1}{2}d\left(x,\mathrm{exp}_{\bar {x}}\left(r\nabla^{\Sigma}u(\bar{x})+r\bar{y}\right)\right)^{2}&=ru(x)+\frac {1}{2}d(x,p)^{2}\\
&\geq ru(\bar{x})+\frac {1}{2}d(\bar{x},p)^{2}\\
&=ru(\bar{\gamma}(0))+\frac{1}{2}|\bar{\gamma}'(0)|^{2}\\
&=ru(\bar{x})+\frac{1}{2}r^{2}\left(|\nabla^{\Sigma}u(\bar{x})|^{2}+|\bar{y}|^{2}\right).
\end{array}
\end{equation*}
\end{proof}

\begin{lemma}	Suppose that $(\bar{x},\bar{y})\in A$, and let $\bar {\gamma}(t):=\mathrm{exp}_{\bar {x}}\left(rt\nabla^{\Sigma}u(\bar{x})+rt\bar{y}\right)$ for $t\in[0,1]$. If $Z$ is a vector field along $\bar{\gamma}$ satisfying $Z(0)\in T_{\bar{x}}\Sigma$ and $Z(1)=0$, then
\begin{equation*}
\begin{array}{lllll}
r(D_{\Sigma}^{2}u)(Z(0),Z(0))-r\langle \Rmnum{2}(Z(0),Z(0)),\bar{y}\rangle\\
+\displaystyle\int_{0}^{1}\left(|\bar{D}_{t}Z(t)|^{2}-\bar{R}(\bar{\gamma}'(t),Z(t),\bar{\gamma}'(t),Z(t)) \right)dt\geq 0.
\end{array}
\end{equation*}
\end{lemma}

\begin{lemma}	Suppose that $(\bar{x},\bar{y})\in A$. Then $g+rD_{\Sigma}^{2}u(\bar{x})-r\langle \Rmnum{2}(\bar{x}),\bar{y}\rangle\geq 0$.
\end{lemma}

\begin{lemma}	Suppose that $(\bar{x},\bar{y})\in A$, and let $\bar {\gamma}(t):=\mathrm{exp}_{\bar {x}}\left(rt\nabla^{\Sigma}u(\bar{x})+rt\bar{y}\right)$ for $t\in[0,1]$. Moreover, let $\{e_{1},...,e_{n}\}$ be an orthonormal basis of $T_{\bar{x}}\Sigma$. Suppose that $W$ is a Jacobi field along $\bar{\gamma}$ satisfying $W(0)\in T_{\bar{x}}\Sigma$ and $\langle \bar{D}_{t}W(0),e_{j}\rangle=r(D_{\Sigma}^{2}u)(W(0),e_{j})-r\langle \Rmnum{2}(W(0),e_{j}),\bar{y}\rangle$ for each $1\leq j\leq n$. If $W(\tau)=0$ for some $0<\tau<1$, then $W$ vanishes identically.
\end{lemma}

\begin{lemma}	The Jacobian determinant of $\Phi$ satisfies
\begin{equation*}
  |\mathrm{det}D\Phi(x,y)|\leq r^{m}\mathrm{det}(g+rD_{\Sigma}^{2}u(x)-r\langle \Rmnum{2}(x),y\rangle)
\end{equation*}
for all $(x,y)\in A$.
\end{lemma}

The proofs of Lemma 3.3-3.6 are identical to Lemma 2.1-2.3 and Lemma 2.5 in \cite{3} respectively. We omit them.

\begin{lemma}	The Jacobian determinant of $\Phi$ satisfies
\begin{equation*}
e^{-\frac{d(x,\Phi(x,y))^{2}}{4r^{2}}}  |\mathrm{det}D\Phi(x,y)|\leq r^{n+m}f(x)e^{\frac {n}{r}-n}e^{-\frac{|2H(x)+y|^{2}}{4}}
\end{equation*}
for all $(x,y)\in A$.
\end{lemma}
\begin{proof}
Given a point $(x,y)\in A$. Using the identity $\mathrm{div}_{\Sigma}\left(f\nabla^{\Sigma}u\right)=f\log f-\frac {|\nabla^{\Sigma}f|^{2}}{f}-f|H|^{2}$, we have
\begin{equation*}
\begin{array}{lllll}
\Delta _{\Sigma}u(x)-\langle H(x),y\rangle&=\log f(x)-\frac{|\nabla^{\Sigma}f(x)|^{2}}{f(x)^{2}}-|H(x)|^{2}\\
&-\frac{\langle\nabla^{\Sigma}f(x),\nabla^{\Sigma}u(x)\rangle}{f(x)}-\langle H(x),y\rangle\\
&=\log f(x)+\frac{|\nabla^{\Sigma}u(x)|^{2}+|y|^{2}}{4}\\
&-\frac{|2\nabla^{\Sigma}f(x)+f(x)\nabla^{\Sigma}u(x)|^{2}}{4f(x)^{2}}-\frac{|2H(x)+y|^{2}}{4}\\
&\leq\log f(x)+\frac{|\nabla^{\Sigma}u(x)|^{2}+|y|^{2}}{4}-\frac{|2H(x)+y|^{2}}{4}.
\end{array}
\end{equation*}
Using Lemma 3.4, Lemma 3.6 and the elementary inequality $\lambda\leq e^{\lambda-1}$, we have
\begin{equation*}
\begin{array}{lllll}
|\mathrm{det}D\Phi(x,y)|&\leq r^{m}\mathrm{det}(g+rD_{\Sigma}^{2}u(x)-r\langle \Rmnum{2}(x),y\rangle)\\
&=r^{n+m}\mathrm{det}(\frac{g}{r}+D_{\Sigma}^{2}u(x)-\langle \Rmnum{2}(x),y\rangle)\\
&\leq r^{n+m}e^{\frac{n}{r}+\Delta_{\Sigma}u(x)-\langle H(x),y\rangle-n}\\
&\leq r^{n+m}e^{\frac{n}{r}+\log f(x)+\frac{|\nabla^{\Sigma}u(x)|^{2}+|y|^{2}}{4}-\frac{|2H(x)+y|^{2}}{4}-n}\\
&=r^{n+m}f(x)e^{\frac {n}{r}-n}e^{-\frac{|2H(x)+y|^{2}}{4}}e^{\frac{d(x,\Phi(x,y))^{2}}{4r^{2}}}.
\end{array}
\end{equation*}
The lemma follows.
\end{proof}

By Lemma 3.2, for any fixed $p\in M$, we choose some point $(x(p),y(p))\in A$ arbitrarily such that $\Phi(x(p),y(p))=p$. Using Lemma 3.2, Lemma 3.7 and area formula \cite{5}, we have
\begin{equation*}
\begin{array}{lllll}
\displaystyle\int_{M}e^{-\frac{d(x(p),p)^{2}}{4r^{2}}}d\mathrm{vol}(p)&\leq\displaystyle\int_{M}\left(\int_{\{\Phi=p\}}e^{-\frac{d(x,\Phi(x,y))^{2}}{4r^{2}}}d\mathscr{H}^{0}\right)d\mathrm{vol}(p)\\
&=\displaystyle\int_{\Sigma}\left( \int_{T_{x}^{\perp}\Sigma}e^{-\frac{d(x,\Phi(x,y))^{2}}{4r^{2}}}|\mathrm{det}D\Phi(x,y)|1_{A}(x,y)dy   \right)d\mathrm{vol}(x)\\
&\leq\displaystyle\int_{\Sigma}\left( \int_{T_{x}^{\perp}\Sigma}r^{n+m}f(x)e^{\frac {n}{r}-n}e^{-\frac{|2H(x)+y|^{2}}{4}}1_{A}(x,y)dy   \right)d\mathrm{vol}(x)\\
&\leq\displaystyle\int_{\Sigma}\left( \int_{T_{x}^{\perp}\Sigma}r^{n+m}f(x)e^{\frac {n}{r}-n}e^{-\frac{|2H(x)+y|^{2}}{4}}dy   \right)d\mathrm{vol}(x)\\
&=r^{n+m}e^{\frac {n}{r}-n}(4\pi)^{\frac{m}{2}}\displaystyle\int_{\Sigma}f(x)d\mathrm{vol}(x),
\end{array}
\end{equation*}
where $\mathscr{H}^{0}$ denotes the counting measure. Using Proposition 2.1, we can divide by $r^{n+m}$ and send $r\rightarrow\infty$ since $M$ satisfies the condition (P). This gives
\begin{equation*}
  (4\pi)^{\frac{n+m}{2}}\theta\leq e^{-n}(4\pi)^{\frac{m}{2}}\int_{\Sigma}f(x)d\mathrm{vol}(x).
\end{equation*}
Consequently,
\begin{equation*}
  n+\frac {n}{2}\log(4\pi)+\log \theta\leq\log {\left(\int_{\Sigma}f d\mathrm{vol}\right)}.
\end{equation*}
Combining this inequality with the normalization
\begin{equation*}
  \int_{\Sigma}f\log f d\mathrm{vol}-\int_{\Sigma}\frac {|\nabla^{\Sigma}f|^{2}}{f}d\mathrm{vol}-\int_{\Sigma}f|H|^{2}d\mathrm{vol}=0
\end{equation*}
gives
\begin{equation*}
\begin{array}{lllll}
&\displaystyle\int_{\Sigma}f\left(\log f+n+\frac {n}{2}\log(4\pi)+\log \theta \right)d\mathrm{vol}-\int_{\Sigma}\frac {|\nabla^{\Sigma}f|^{2}}{f}d\mathrm{vol}-\int_{\Sigma}f|H|^{2}d\mathrm{vol}\\
&=\displaystyle\int_{\Sigma}f\left(n+\frac {n}{2}\log(4\pi)+\log \theta \right)d\mathrm{vol}\\
&\displaystyle\leq\left(\int_{\Sigma}f d\mathrm{vol}\right)\log {\left(\int_{\Sigma}f d\mathrm{vol}\right)}.
\end{array}
\end{equation*}

It remains to consider the case when $\Sigma$ is disconnected. For completeness, we list Brendle's proof \cite{2}. In that case, we apply the inequality to each individual connected component of $\Sigma$, and sum over all connected components. Since
\begin{equation*}
  a\log{a}+b\log{b}<a\log{(a+b)}+b\log{(a+b)}=(a+b)\log{(a+b)}
\end{equation*}
for $a,b>0$, we conclude that
\begin{equation*}
\begin{array}{lllll}
&\displaystyle\int_{\Sigma}f\left(\log f+n+\frac {n}{2}\log(4\pi)+\log \theta \right)d\mathrm{vol}-\int_{\Sigma}\frac {|\nabla^{\Sigma}f|^{2}}{f}d\mathrm{vol}-\int_{\Sigma}f|H|^{2}d\mathrm{vol}\\
&\displaystyle<\left(\int_{\Sigma}f d\mathrm{vol}\right)\log {\left(\int_{\Sigma}f d\mathrm{vol}\right)}
\end{array}
\end{equation*}
if $\Sigma$ is disconnected. This completes the proof of Theorem 1.1.

\section{Proof of Corollary 1.2}
By computing, the volume form $d\mathrm{vol}(u,v)$ is equal to $u\sqrt{1+a^{2}u^{2}}dudv$, and the intrinsic distance from the Origin to the point $\overrightarrow{r}(u,v)$ satisfies
\begin{equation*}
\begin{array}{lllll}
d(O,\overrightarrow{r}(u,v))&=\displaystyle\int_{0}^{u}\sqrt{1+a^{2}t^{2}}dt\\
&=\frac{u}{2}\sqrt{1+a^{2}u^{2}}+\frac{1}{2a}\ln{(au+\sqrt{1+a^{2}u^{2}})}.
\end{array}
\end{equation*}
So we have
\begin{equation*}
\displaystyle\int_{\Gamma}e^{-\frac{d(O,\overrightarrow{r}(u,v))^{2}}{4r^{2}}}d\mathrm{vol}(u,v)=2\pi\int_{0}^{\infty}e^{-\frac{A(u)^{2}}{4r^{2}}}u\sqrt{1+a^{2}u^{2}}du,
\end{equation*}
where $A(u)=d(O,\overrightarrow{r}(u,v))$. Since $u\leq 2A(u)$ for all $u\geq 0$, we have
\begin{equation*}
\begin{array}{lllll}
\displaystyle\int_{0}^{\infty}e^{-\frac{A(u)^{2}}{4r^{2}}}u\sqrt{1+a^{2}u^{2}}du&\leq \displaystyle\int_{0}^{\infty}e^{-\frac{A(u)^{2}}{4r^{2}}}2A(u)\sqrt{1+a^{2}u^{2}}du               \\
&=\displaystyle\int_{0}^{\infty}e^{-\frac{t}{4r^{2}}}dt\\
&=4r^{2}.
\end{array}
\end{equation*}
Note that $\lim_{u\rightarrow 0^{+}}A(u)=\frac{1}{2}u$. Thus, for any $\varepsilon>0$, we can find a positive number $\delta=\delta(\varepsilon)$ such that $u\geq \frac{2}{1+\varepsilon}A(u)$ for all $u\in[0,\delta]$. So we have
\begin{equation*}
\begin{array}{lllll}
\displaystyle\int_{0}^{\infty}e^{-\frac{A(u)^{2}}{4r^{2}}}u\sqrt{1+a^{2}u^{2}}du&\geq\displaystyle\int_{0}^{\delta}e^{-\frac{A(u)^{2}}{4r^{2}}}u\sqrt{1+a^{2}u^{2}}du\\
&\geq\frac{2}{1+\varepsilon}\displaystyle\int_{0}^{\delta}e^{-\frac{A(u)^{2}}{4r^{2}}}A(u)\sqrt{1+a^{2}u^{2}}du\\
&=\frac{4r^{2}}{1+\varepsilon}\displaystyle\int_{0}^{\frac{A(\delta)^{2}}{4r^{2}}}e^{-t}dt\\
&=\frac{4r^{2}}{1+\varepsilon}(1-e^{-\frac{A(\delta)^{2}}{4r^{2}}}).
\end{array}
\end{equation*}
And we can find a positive number $N=N(\varepsilon)$ such that $e^{-\frac{A(\delta)^{2}}{4r^{2}}}<\varepsilon$ for all $r>N$. From standard $\varepsilon-\delta$ language, we can conclude that $\theta=2$. Using Theorem 1.1, the corollary follows.

\section{Acknowledgement}
The first named author thanks Professor Yu Zheng and his classmate Yukai Sun for helpful discussions.

\end{document}